\documentclass[reqno,11pt]{amsart}
\usepackage{amsmath,amssymb}
\usepackage{graphicx}
\usepackage[top=1in,bottom=1in,left=1in,right=1in]{geometry}
\usepackage{dsfont}
\newtheorem{thm}{Theorem}

\newtheorem{lemma}{Lemma}

\theoremstyle{definition}
\newtheorem{rem}{Remark}[section]

\usepackage{color}

\newcommand{\eps}{\varepsilon}

\def\R{\mathbb R}
\def\N{\mathbb N}
\def\rife#1{(\ref{#1})}

\newcommand{\be}{\begin{equation} \label}
\newcommand{\ee}{\end{equation}}

\newcommand{\ts}{\textstyle}

\begin{document}
\title[] 
{Classification of solutions of an \\ elliptic Hamilton-Jacobi equation}


\author[Porretta]{Alessio Porretta}%
\address{University of Rome  Tor Vergata,
Dipartimento di Matematica,
Via della Ricerca Scientifica 1,
00133 Roma, Italia}
\email{porretta@mat.uniroma2.it}

\author[Souplet]{Philippe Souplet}%
\address{Universit\'e Sorbonne Paris Nord,
CNRS UMR 7539, Laboratoire Analyse, G\'{e}om\'{e}trie et Applications,
93430 Villetaneuse, France}
\email{souplet@math.univ-paris13.fr}

\begin{abstract}
 We show that any classical solution of the diffusive Hamilton-Jacobi (DHJ) equation $-\Delta u= |\nabla u|^p$
in a half-space with zero boundary conditions for $1<p\le 2$ is necessarily one-dimensional.
This improves the previously known result, which required an extra assumption of boundedness from above.
Combined with the existing analogous result for $p>2$, our result completes the full classification picture
of the Dirichlet problem for equation (DHJ) in a half-space.
\end{abstract}

\maketitle



\section{Introduction and main result}

In this paper, we are concerned with the classification of solutions of the elliptic Hamilton-Jacobi equation 
\be{ellHJ}
-\Delta u= |\nabla u|^p
\ee
in the half-space.

 Liouville type nonexistence, classification and symmetry (or rigidity) theorems are a central issue 
in the study of nonlinear elliptic and parabolic problems, as well as in other nonlinear PDE's.
Beside their intrinsic interest, they have, in conjunction with rescaling methods, many applications 
for the description of the qualitative behavior of solutions (a priori estimates, space and/or time decay, blow-up 
asymptotics, etc.).
The most studied cases of spatial domains are the whole space $D=\R^n$, the half-space 
$D=\R^n_+:=\{x\in\R^n;\, x_n>0\}$ and cones. The first two cases are especially meaningful 
for applications since they arise as limiting domains when studying a priori estimates via rescaling methods
(see, e.g., \cite{GS1,PQS} and \cite{QSb19} and the references therein).

Many Liouville type nonexistence, classification and symmetry results 
are available for elliptic equations.
In order to place our results in perspective,
let us review the relevant known facts.
In all the paper, by a solution,  unless stated otherwise, we always mean 
a classical solution,
namely $u\in C^2(\R^n)$ in the whole space case, and $u\in C^2(\R^n_+)\cap C(\overline{\R^n_+})$ in the half-space case. 

Let us first consider the well-known Emden equation ($p>1$)  {in the whole space:}
\be{ellclass}
-\Delta u= u^p,\quad x\in\R^n\qquad  {(u>0).}
\ee

	\begin{itemize}
\item[$\bullet$] The classical result of \cite{GS1} asserts that equation \eqref{ellclass} has no solution
if (and only if) 
$p<p_S:=(n+2)/(n-2)_+$ \cite{GS1}.

\vskip 1.5pt

\item[$\bullet$]  {For the critical case} $p=p_S$,  it was shown in \cite{CGS} that, 
up to dilations and translations, the only 
solution of \eqref{ellclass} is the Aubin-Talenti bubble $u(x)=c(n)(1+|x|^2)^{2-n\over 2}$ .

	\end{itemize}
	
\vskip 1.5pt

Next consider the corresponding half-space problem:
\be{ellclass2}
\begin{cases}
-\Delta u = u^p,  & \hbox{$x\in \R^n_+$}\qquad (u>0),
\\
u=0, & \hbox{$x\in \partial \R^n_+$.}
\end{cases}
\ee

\begin{itemize}
\item[$\bullet$] It was first shown in \cite{GS2} that \eqref{ellclass2} has no 
solution whenever $p\le p_S$.
A lot of effort was then devoted to removing the restriction~$p\le p_S$:

\vskip 1.5pt

\item[$\bullet$] For {\it bounded} solutions, the nonexistence of 
solutions of \eqref{ellclass2} was proved in \cite{CLZ} {\it for all $p>1$}, 
improving on earlier partial results from \cite{Da,Fa}.
The boundedness assumption was next weakened in  \cite{DSS} where it was shown
that, for all $p>1$, problem \eqref{ellclass2} has no {\it monotone} solutions (i.e., solutions that are increasing in the $x_n$-direction),
and it is known that monotonicity is true in particular for any solution that is bounded on finite strips.
The result was then extended in \cite{DFP} to the larger class of solutions that are stable outside a compact set.
However, for $p>p_S$, the validity of the full nonexistence property (for all positive classical solutions) still remains an open problem.	
\end{itemize}

\smallskip

  Now passing to our main topic, namely the elliptic Hamilton-Jacobi equation \eqref{ellHJ}, with $p>1$,
  the following is known:

	\begin{itemize}
\item[$\bullet$] If $u$ is a classical solution of \eqref{ellHJ} in $\R^n$, 
then $u$ is constant \cite{Lions85}.
	\end{itemize}
	
\noindent 	For the half-space problem:
\be{pb}
\begin{cases}
-\Delta u { \strut=\strut}  |\nabla u|^p,  & \hbox{$x\in \R^n_+$,}
\\
u=0, & \hbox{$x\in \partial \R^n_+$}
\end{cases}
\ee
we have:

	\begin{itemize}
	
\item[$\bullet$] If $p>2$ and $u$ is a classical solution of \eqref{pb},
then $u$ depends only on the variable $x_n$ (see \cite[Theorem~1.1]{FPS2020}).
Applications of this classification result to the description of gradient blow-up asymptotics for the corresponding
 initial-Dirichlet parabolic problem were also developed in \cite{FPS2020}.

\vskip 1.5pt

\item[$\bullet$] If $p\in(1,2]$, $u$ is a classical solution of \eqref{pb} and $u$ is bounded from above,
then $u$ depends only on the variable $x_n$ (see \cite[Theorem~4.1]{PV}).
However, the following question was left open:
\be{OQ}
\begin{aligned}
&\hbox{Does the one-dimensionality property for \eqref{pb} with $p\in(1,2]$}\\
&\hbox{remain true without boundedness assumption~?}
\end{aligned}
\ee

	\end{itemize}

\begin{rem}
Beside equations and \eqref{ellHJ} and \eqref{ellclass}, the mixed equations
$-\Delta u= u^q|\nabla u|^p$ and $-\Delta u=u^q+\mu |\nabla u|^p$ ($p,q>1$, $\mu\ne 0$)
have also been studied and a number of
Liouville-type results (depending on the various parameters $p,q,n,\mu$) can be found in,
e.g.,~\cite{BZZ,BV21,BGV,BGV25,FPS2020b,PV} and the references therein.
 {See also \cite{BQT25, BGQ, CHZ,CG} for results on more general, related equations.}
\end{rem}

The main goal of this paper is to solve the open question \eqref{OQ}.
Namely, we will answer it positively and hence complete the full classification picture for all~$p>1$.

\medskip

Our main result is the following theorem.

\begin{thm} \label{thmclassif}
Let  $1<p{\strut\le\strut} 2$ and let $u\in C^2(\R^n_+)\cap C(\overline{\R^n_+})$ be a classical solution of \rife{pb}. Then $u$ depends only on $x_n$.
\end{thm}

 As a consequence, all solutions of problem \rife{pb} can be explicitly obtained by a simple ODE integration.
We stress that no sign condition or growth assumption at infinity 
 is assumed a priori on $u$.

The proof of the corresponding result of \cite{PV} crucially required the assumption that $u$ be bounded from above.
In fact the result is stated there for nonnegative solutions of the equation $-\Delta v+|\nabla v|^p=0$ in $\R^n_+$
with  {constant boundary conditions $v=C$.\footnote{which is equivalent to \rife{pb} with $M:=\sup_{\R^n_+} u<\infty$, 
through the transformations $u:=C-v$ and $v:=M-u$}}
As for the superquadratic case $p>2$, solved in \cite{FPS2020} without boundedness assumption via a moving planes
and translation-compactness argument, 
it is in a sense easier because the Bernstein estimate from \cite{Lions85} (see Lemma \ref{bernstein} below)
implies the crucial property that $u(x',x_n)\to 0$ as $x_n\to 0$ {\it uniformly} with respect to $x'\in\R^{n-1}$.
This property no longer follows from the Bernstein estimate when $p\in(1,2]$.

The key new idea is to show the global boundedness of $u$ (Lemma~\ref{barrier2}) by a rescaling procedure.
The latter is combined with an oscillation shrinking argument
 (using Lemma~\ref{lemDoubling}, which is a variant of the doubling lemma from \cite{PQS}) 
 and a boundary barrier estimate (Lemma~\ref{barrier}).
This forces the rescaled solution to achieve an interior maximum,
leading to a contradiction with the strong maximum principle.
 Once the boundedness is established, we can either apply \cite[Theorem~4.1]{PV} or conclude via a moving planes
 and translation-compactness argument similar to that in \cite{FPS2020}.

	 \section{Proof of Theorem~\ref{thmclassif}}

\medskip

Throughout this paper we denote
$$\beta=\frac{1}{p-1}.$$
Let us first recall the classical gradient estimate from \cite{Lions85}, 
obtained there by Bernstein's method,  which, in the special case of \rife{pb}, can be stated as follows:

\begin{lemma}\label{bernstein}
Under the assumptions of Theorem~\ref{thmclassif}, there exists  a universal constant $K$, only depending on $p,n$, such that
\be{bernsteinK}
|\nabla u(x)| \leq Kx_n^{-\beta},\quad x\in \R^n_+\,.
\ee
\end{lemma}
\qed

{For $1<p<2$,} as a consequence of the gradient estimate, we next observe that $u$ is bounded away from $x_n=0$ and admits a limit as $x_n\to \infty$.

\begin{lemma}\label{bernstein2} 
Under the assumptions of Theorem~\ref{thmclassif} {with $1<p<2$}, there exists $\ell\in \R$ such that 
\be{ell}
\lim_{x_n\to \infty} u(x',x_n) = \ell \qquad \hbox{uniformly for $x'\in \R^{n-1}$,}
\ee
and
\be{estu}
|u(x)| \leq |\ell| + C(n,p)x_n^{1-\beta}, \quad x\in \R^n_+\,.
\ee
\end{lemma}

\begin{proof} 
Since, due to Lemma \ref{bernstein},
$$
|u(x',t)-u(x',s) |= \Big|\int_s^t \frac{\partial u}{\partial x_n}(x',r)dr\Big| \leq K \int_s^t r^{-\beta}dr\,,
$$
and since $\beta>1$ owing to $1<p<2$, it follows that $u(x',x_n)$ admits a finite limit as $x_n\to \infty$. Moreover, since $|\nabla u|\leq K x_n^{-\beta}\to 0$ as $x_n\to \infty$, this limit is independent of $x'$.
This implies \rife{ell} and, by integration, \rife{estu} holds as well.
\end{proof}

We next state a local boundary estimate in terms of a local bound on $u$.
For $a=(a',0)\in\partial\R^n_+$ we denote
$$D_a=\{x\in \R^n_+;\ |x-a|<2\},\qquad
\Sigma_a=\{x\in \R^n;\ |x-a|<2,\ x_n=0\}.$$

\begin{lemma}\label{barrier}
Let $1<p{\strut\le\strut} 2$, $M>0$ and $a=(a',0)\in\partial\R^n_+$.
Let $u\in C^2(D_a)\cap C(\overline D_a)$  be a solution of $-\Delta u {\strut\le\strut}  |\nabla u|^p$ in $D_a$ 
with $u=0$ on $\Sigma_a$, and
such that ${ u}\le M$ in $\overline D_a$.
There exists $M_1>0$ depending only on $M,n,p$ such that
$${u}(a',x_n)\le M_1 x_n,\quad 0<x_n\le 1.$$
\end{lemma}

This follows from a barrier argument valid for $1<p\le 2$ (but not for $p>2$);
cf.~\cite[Lemma~VI.3.1]{LSU} and also \cite[Lemma 35.4]{QSb19}. We give a proof for convenience.

\begin{proof}
Let $U$ be the solution of 
$$ \left.\qquad{\alignedat2
 -\Delta U &= 1,    &\qquad& 1<|x|<3,\\
             U &= 0,  &\qquad& |x|=1, \\
             U &= 1,  &\qquad& |x|=3. \\
  \endalignedat}\qquad\right\} 
$$
By the maximum principle and the smoothness of $U$ we have, for some $c_2=c_2(n)>0$,
\be{SauxilBernsteinIneq}
0<U(x)<1 \quad\hbox{ and }\quad  U(x)\leq c_2(|x|-1),\qquad 1<|x|<3.
\ee
Let $b=(a',-\ts\frac12)$ and define the annulus $\mathcal{A}=\{x\in\R^n;\, \ts \frac12\leq|x-b|\leq \ts\frac32\}$,
whose inner ball is tangent to $\partial\R^n_+$ from outside at the point $a$,
and which satisfies $\mathcal{A}\cap\R^n_+\subset D_a$. Set
$$V(x)=\log\bigl(1+e^{M}U\big(2(x-b)\bigr),
\quad x\in \mathcal{A}.$$
We will compare $V$ with $u$ in the cap $\Omega=\mathcal{A}\cap\R^n_+=\mathcal{A}\cap D_a$. 
Since $e^V=1+e^{M}U\big(2(x-b)\big)\le 2e^{M}$, we obtain
$-\Delta V=|\nabla V|^2+4e^{-V}e^{M}\ge |\nabla V|^2+2\ge |\nabla V|^p$,
hence
$$ -\Delta V-|\nabla V|^p\ge 0\ge -\Delta u-|\nabla u|^p
\quad\hbox{in $\Omega$.}$$
We have 
$\partial\Omega=S_1\cup S_2$, with $S_1=\{x\in\R^n_+;\,|x-b|=\frac32\}$
and $S_2=\{x\in \R^n;\ |x-b|<\frac32,\ x_n=0\}$.
Since $V\ge M\geq {u}$ on $S_1$
and $V\ge 0=u$ on $S_2$,
it follows from the maximum principle that $V\ge u$ in $\Omega$.
In particular, using \eqref{SauxilBernsteinIneq} and $\log(1+s)\le s$ for $s>0$, we obtain
$$u(a',x_n)\leq V(a',x_n) \leq e^M U\big(0,2x_n+1\big)
\leq 2c_2e^M x_n
\quad 0<x_n\le 1,$$
which completes the proof.
\end{proof}

We shall also use the following modification of (a special case of) the doubling lemma in \cite[Lemma~5.1]{PQS}, which will enable us  to carry out an oscillation shrinking argument in the subsequent step.

\begin{lemma} \label{lemDoubling}
Let $M:\R^n\to[0,\infty)$ be bounded on compact sets
and let $k>0$ and $\rho>1$ be real numbers. If $y\in \R^n$ is such that $M(y)>0$,
then there exists $x\in \R^n$ such that
\begin{equation}
M(x)\geq M(y),
                \label{DoublingB}\end{equation}
and
$$M(z)\le \rho M(x)\quad\hbox{ for all $z\in \overline
B\bigl(x,\frac{k}{M(x)}\bigr)$}.$$
\end{lemma}

\begin{proof}
Assume that the Lemma is not true. 
We claim that there exists  
a sequence $(x_j)$ in $\R^n$
such that, for all $j\in\N$,
\be{MinductB}
 M(x_{j+1})>\rho M(x_j)           
\quad\hbox{and}\quad 
|x_j-x_{j+1}|\leq \ts\frac{k}{M(x_j)}.           
\ee
We choose $x_0=y$. By our contradiction assumption,
there exists $x_1\in \R^n$ such that
$$M(x_1)>\rho M(x_0)
\quad\hbox{and}\quad 
|x_0-x_1|\leq \ts\frac{k}{M(x_0)}.$$
Fix some $i\geq 1$ and assume that we have already constructed 
$x_0,\cdots,x_i$
such that \eqref{MinductB} holds for $j=0,\cdots,i-1$.
Since $M(x_i)\ge M(y)$, our contradiction assumption
implies  that there exists $x_{i+1}\in \R^n$ such that
$$M(x_{i+1})>\rho M(x_i)\quad\hbox{and}\quad 
|x_i-x_{i+1}|\leq \ts\frac{k}{M(x_i)}.$$
We have thus proved the claim by induction.

Now, we have
\begin{equation}
 M(x_i)\geq \rho^iM(x_0)\quad\hbox{and}\quad 
|x_i-x_{i+1}|\leq \rho^{-i}\ts\frac{k}{M(x_0)},\qquad i\in\N.
                 \label{Mdiverge}\end{equation}
In particular, $(x_i)$ is a Cauchy sequence,
hence it converges to some $a\in \R^n$
and $K:=\{x_i;\ i\in\N\}\cup\{a\}$ is thus a compact subset of $\R^n$.
Since $M(x_i)\to\infty$ as $i\to\infty$ by \eqref{Mdiverge},
this contradicts the assumption that $M$ is bounded on compact sets. The Lemma is proved. \end{proof}

The key step
 is the following lemma, which shows the global boundedness of $u$.
 Its proof is based on a rescaling procedure, combined with an oscillation shrinking argument
 (using Lemma~\ref{lemDoubling}) and a boundary estimate (from Lemma~\ref{barrier}), which force the rescaled solution to achieve an interior maximum,
 leading to a contradiction with the strong maximum principle.

\begin{lemma}\label{barrier2}
Under the assumptions of Theorem~\ref{thmclassif} {with $1<p<2$}, 
we have $\sup_{\R^n_+}|u|<\infty$.
\end{lemma}

\begin{proof} 
In this proof $C$ denotes a generic positive constant possibly depending on $u$.
Assume for contradiction that  $\sup_{\R^n_+}|u|=\infty$.
Then there exists a sequence $y^k\in \R^n_+$ such that $\lim_{k\to\infty}  |u(y^k)|=\infty$.
{Note that $\beta>1$ and} let $M(x)=|u(x)|^{1/(\beta-1)}$, extended by $0$ in $\R^n\setminus \overline{\R^n_+}$.
For every integer $k\ge 1$, applying Lemma~\ref{lemDoubling} with $\rho=2^{1/k}$,
there exists a point $x^k\in \R^n_+$ such that $M(x^k)\ge M(y^k)$ and
\be{doublingk}
M(x)\le 2^{1/k} M(x^k)\quad\hbox{ for all $x\in \R^n$ such that $|x-x^k|\le \ts\frac{k}{M(x^k)}$.}
\ee
We rescale $u$ by setting:
$$v_k(y)=\lambda_k^{\beta-1}u(x^k+\lambda_k y),\quad y\in\tilde D_k,$$
where 
$$\lambda_k:=\ts\frac{1}{M(x^k)}=|u(x^k)|^{-1/(\beta-1)},\quad \eta_k:=\lambda_k^{-1}x^k_n,$$
$$\tilde D_k:=\{y\in\R^n,\  y_n>-\eta_k\},\quad
\tilde \Sigma_k:=\{y\in\R^n,\  y_n=-\eta_k\}.$$
The function $v_k$ solves
$$\begin{cases}
-\Delta v_k = |\nabla v_k|^p,  & \hbox{$y\in\tilde D_k$,}
\\
v_k=0, & \hbox{$y\in \Sigma_k$}
\end{cases}
$$
and
\be{normalizek}
|v_k(0)|=1.
\ee
Moreover, by \eqref{doublingk}, we get
\be{doublingk2}
|v_k(y)|\le 2^{\frac{\beta-1}{k}},\quad y\in\tilde D_k\cap B_k(0).
\ee
Since $\lim_{k\to\infty}  |u(x^k)|=\infty$ and $|u(x^k)|\le C(1+(x_n^k)^{1-\beta})$
owing to \eqref{estu}, it follows that $\lim x^k_n=0$. Therefore,
$$\eta_k=|u(x^k)|^{1/(\beta-1)}x^k_n\le C(1+x^k_n)\le C.$$
By passing to a subsequence we may assume that $\eta_k\to L\in[0,\infty)$.

Now applying Lemma~\ref{barrier} {to $v_k$ and $-v_k$, we deduce that} there exist $M_1>0$ and $k_0\ge 1$ such that, for all $k\ge k_0$,
\be{estimbarrier}
|v_k(y',-\eta_k+s)|\le M_1s,\quad |y'|<k/2,\ 0{\strut\le\strut}s<1.
\ee
It follows in particular that $L>0$. Indeed, if we had $L=0$, then taking $s=\eta_k$ in \eqref{estimbarrier} with $k$ large and using \rife{normalizek}
would lead to $1\leq M_1\eta_k\to 0$ as $k\to \infty$: a contradiction.
In addition, by Lemma \ref{bernstein}, we have 
\be{estimgradvk}
|\nabla v_k(y',-\eta_k+s)|\le Ks^{-\beta},
\quad y'\in \R^{n-1},\ s>0.
\ee
Let $D_\infty:=\{y\in\R^n,\  y_n>-L\}$.
By \eqref{estimbarrier}, \eqref{estimgradvk} and interior elliptic estimates, passing to a subsequence, it follows that $v_k$ converges in $C^2_{loc}(D_\infty)$
to a solution $v$ of $-\Delta v = |\nabla v|^p$ in $D_\infty$.
Moreover, \eqref{estimbarrier} guarantees that $|v(y',-L+s)|\le M_1s$ for all $y'\in \R^{n-1}$ and $s\in(0,1)$,
hence $v\in C(\overline{D_\infty})$ with 
\be{BCv}
v=0 \quad\hbox{on $\partial D_\infty$.}
\ee
Also, by \eqref{normalizek}, \eqref{doublingk2}, we have $|v|\le 1$ in $D_\infty$
and $|v(0)|=1$.
Since $L>0$, we have that $0$ is an interior maximum or minimum point of $v$ in $D_\infty$.
It then follows from the strong maximum principle that $v\equiv \pm 1$ in $\overline{D_\infty}$:
a contradiction with \eqref{BCv}.
\end{proof}

\begin{proof} [Proof of Theorem~\ref{thmclassif} {for $1<p<2$}]
In view of Lemma~\ref{barrier2}, this is a consequence of 
\cite[Theorem~4.1]{PV} applied to $M-u\ge 0$ where $M=\sup_{\R^n_+}u$.

Alternatively, we can provide a different, self-contained proof as follows.
By combining Lemma~\ref{barrier}, {applied to $u$ and $-u$,} and Lemma~\ref{barrier2},
we obtain a uniform Lipschitz control of $u$ near the boundary, namely:
\be{boundvM1}
|{ u}(x',x_n)|\le M_1 x_n,\quad x'\in \R^n_+,\ 0<x_n\le 1,
\ee
for some $M_1>0$.
This allows us to repeat the argument from \cite{FPS2020}, used there to show the analogous result for $p>2$.
We give the details for convenience.

Write $x=(\tilde x,y)\in \mathbb R^{n-1}\times  [0,\infty)$ and
fix any $h\in \mathbb R^{n-1}\setminus\{0\}$.  Let
$$v(\tilde x,y)=u(\tilde x+h,y)-u(\tilde x,y),\qquad (\tilde x,y)\in \mathbb R^{n-1}\times [0,\infty).$$
It suffices to show that $v\equiv 0$.
By Lemma~\ref{barrier2} we know that $v$ is bounded.
Assume for contradiction that
$\sigma:=\sup_{\mathbb R^n_+} v>0$
(the case $\inf_{\mathbb R^n_+} v<0$ is similar). 
By Lemma \ref{bernstein}, we have
$$|v(\tilde x,y)|\le C(n,p)|h|y^{-\beta},\quad\hbox{for all $(\tilde x,y)\in \mathbb R^{n-1}\times (0,\infty)$.}$$
This combined with \eqref{boundvM1} provides some large $A>1$ and small $\eps\in(0,1)$ such that $|v|\le \sigma/2$ 
in $\{y\ge A\}\cup\{y\le \eps\}$.
Therefore
\be{eqContrad2}
\sigma=\sup_{\mathbb R^{n-1}\times (\eps,A)} v.
\ee
Pick a sequence $(\tilde x_j,y_j)\in \mathbb R^{n-1}\times (\eps,A)$ 
such that $v(\tilde x_j,y_j)\to\sigma$.

\smallskip
Next define
$$u_j(\tilde x, y)=u(\tilde x_j+\tilde x,y),\qquad (\tilde x,y)\in \mathbb R^{n-1}\times [0,\infty),$$
and note that
\be{supduj}
\sup_{(\tilde x, y)\in\mathbb R^n_+} \bigl(u_j(\tilde x+h, y)-u_j(\tilde x, y)\bigr)=\sup_{\mathbb R^n_+} v=\sigma
\ee
and
\be{limsigma}
u_j(h,y_j)-u_j(0,y_j) =v(\tilde x_j,y_j) \to \sigma,\quad\mbox{ as }j\to\infty.
\ee
By Lemma \ref{bernstein} and \eqref{estu}, we have
$$|\nabla u_j(\tilde x,y)|\le Ky^{-\beta}, \quad\hbox{for all $(\tilde x,y)\in \mathbb R^{n-1}\times (0,\infty)$.}$$
and
\be{Bernsteiny_uj}
|u_j(\tilde x,y)|\le C(1+y^{1-\beta}),\quad\hbox{for all $(\tilde x,y)\in \mathbb R^{n-1}\times [0,\infty)$}
\ee
for all $j$.
It then follows from interior elliptic estimates that $(u_j)_j$ is relatively compact in $C^2_{\mathrm{loc}}(\mathbb R^n_+)$.
Therefore, some subsequence of $(u_j)_j$ converges in that topology to a solution $U\in C^2(\mathbb R^n_+)$ of
$-\Delta U=|\nabla U|^p$.
Moreover, we may assume that $y_j\to y_\infty\in [\eps,A]$ and we get
\be{supU}
U(h,y_\infty)-U(0,y_\infty)= \sigma,
\ee
owing to \eqref{limsigma}. 

\smallskip
 Put now
$$V(\tilde x,y)=U(\tilde x+h,y)-U(\tilde x,y),\qquad (\tilde x,y)\in \mathbb R^{n-1}\times [0,\infty).$$
It follows from \eqref{supduj} and \eqref{supU} that $\sigma=\sup_{\mathbb R^n_+} V=V(0,y_\infty)$.
But $V$ satisfies
$$-\Delta V=A(\tilde x,y)\cdot\nabla V,
\quad\hbox{where }
A(\tilde x,y):=\int_0^1 G\bigl(s\nabla U(\tilde x+h,y)+(1-s)\nabla U(\tilde x,y)\bigr)\,ds,
$$
with $G(\xi)= p|\xi|^{p-2}\xi$, and $A$ is bounded on compact subsets of $\mathbb R^n_+$.
This contradicts the strong maximum principle and completes the proof.
\end{proof}

 We note that the above proof cannot be directly extended to the remaining case $p=2$,
 mostly because \eqref{bernsteinK} no longer yields estimates \eqref{ell}-\eqref{estu} when $p=2$.
However, by means of the Hopf-Cole transformation and Lemma~\ref{barrier}, this case can be reduced to known properties of 
positive harmonic functions in a half-space. 

\begin{proof} [Proof of Theorem~\ref{thmclassif} for $p=2$]
The function $w:=e^u-1$ is harmonic in $\R^n_+$, with $w=0$ on $\partial \R^n_+$.
Since $w\ge -1$, we deduce from Lemma~\ref{barrier} applied to $-w$ that there exists $k>0$ such that  $w\ge -kx_n$ in 
$\R^{n-1}\times(0,1]$.
Since also $w\ge -1\ge -x_n$ in $\R^{n-1}\times(1,\infty)$, it follows that 
$z:=w+(k+1)x_n$ is harmonic and positive in $\R^n_+$, with $z=0$ on $\partial \R^n_+$.
It is known that any such function is necessarily of the form $z=cx_n$ for some $c>0$
(cf.~\cite[Theorem~1.7.3 and Exercise~1.17]{AG}, and see also 
\cite[Remark~5.1(ii)]{MSS} for a more elementary proof). 
Consequently, $e^u=1+w=1+(c-k-1)x_n$, which proves Theorem~\ref{thmclassif} for $p=2$.
\end{proof}

{\bf Acknowledgement.}\quad A. Porretta was supported by  the Excellence Project MatMod@TOV of the Department of Mathematics of the University of Rome Tor Vergata, by the Italian (EU Next Gen) PRIN project 2022W58BJ5 ({\it PDEs and optimal control methods in mean field games, population dynamics and multi-agent models}, CUP E53D23005910006), and by GNAMPA research group of Indam. He also wishes to thank the hospitality of University of Paris Nord, where this research was initiated. 

\end{document}